\documentclass[12pt,english,a4paper]{amsart}

\usepackage[utf8]{inputenc}

\usepackage{amsfonts}
\usepackage{amssymb}
\usepackage{amsmath}
\usepackage{amsthm}
\usepackage{geometry}
\usepackage{enumerate}
\usepackage{hyperref}

\begin{document}

\newtheorem{theorem}{Theorem}[section]
\newtheorem{lemma}[theorem]{Lemma}
\newtheorem{corollary}[theorem]{Corollary}

\theoremstyle{remark}
\newtheorem*{remark}{Remark}

\title[A family with monodromy group $\mathrm{PSL}_6(2)$]{A family of 4-branch-point covers with monodromy group $\mathrm{PSL}_\textbf{6}\textbf{(2)}$}

\author{Dominik Barth}\email{dominik.barth@mathematik.uni-wuerzburg.de}
\address{Institute of Mathematics, University of W\"urzburg, Emil-Fischer-Stra{\ss}e 30, 97074 W\"urzburg, Germany}

\author{Andreas Wenz}\email{andreas.wenz@mathematik.uni-wuerzburg.de} 
\address{Institute of Mathematics, University of W\"urzburg, Emil-Fischer-Stra{\ss}e 30, 97074 W\"urzburg, Germany}

\begin{abstract}
We describe the explicit computation of a family of 4-branch-point rational functions of degree 63 with monodromy group $\mathrm{PSL}_6(2)$.
This, in particular, negatively answers a question by J.~König whether there exists a such a function with rational coefficients.
The computed family also gives rise to non-regular degree-126 realizations of $\mathrm{Aut}(\mathrm{PSL}_6(2))$ over $\mathbb{Q}(t)$.
\end{abstract}
\maketitle

\section{Introduction}
Let $C:=(C_1,C_2,C_3,C_3)$ be the genus-0 class vector of $\mathrm{PSL}_6(2)$ in its natural 2-transitive action on the 63 non-zero elements of ${\mathbb{F}_2}^6$, where $C_1$, $C_2$ and $C_3$ are the unique conjugacy classes of cycle structure $2^{28}.1^7$, $2^{16}.1^{31}$ and  $3^{20}.1^3$, respectively.
Then, using the theory of Hurwitz spaces J.~König~\cite[p.~109]{Koe} established the theoretical existence of a hyperelliptic genus-$3$ curve $\mathcal{H}$ defined over $\mathbb{Q}$ and  polynomials $p,q\in \mathbb{Q}(\mathcal{H})[X]$ satisfying the following:
\begin{enumerate}
\item[$\bullet$] The family $\mathcal{F}$ of normalized covers with ramification locus $(0,\infty,1+\sqrt{\lambda},1-\sqrt{\lambda})$ where $\lambda \in \mathbb{P}^1\setminus \{0,1,\infty\}$ and ramification structure $C$ can be parameterized by a rational function
$
F = \frac{p}{q} \in \mathbb{Q}(\mathcal{H})(X).
$
\item[$\bullet$] $\mathrm{Gal}(p-tq \mid \mathbb{Q}(\mathcal{H})(t) ) \cong \mathrm{PSL}_6(2).$
\end{enumerate}

In order to decide whether $\mathrm{PSL}_6(2)$ occurs regularly as a Galois group over $\mathbb{Q}(t)$ with ramification structure $C$, one has to check the existence of $\mathbb{Q}$-rational points on $\mathcal{H}$ that lead to Galois group preserving specializations. König also mentions that without explicit computation of $\mathcal{H}$ there seems to be no way of finding an answer to this question.

Note that $\mathrm{PSL}_6(2)$ and $\mathrm{PSp}_6(2)$ of degree $63$ are expected to be the largest (with respect to the permutation degree) almost simple primitive groups having a generating genus-0 tuple of length at least 4 with the socle being a simple group of Lie type. While multi-parameter families of polynomials with Galois group $\mathrm{PSp}_6(2)$ of degree $28$ and $36$ were calculated in \cite{BKW}, the case $\mathrm{PSL}_6(2)$ remained open. 
With the recent development in computing multi-branch-point covers in \cite{BKW} we are able to give explicit defining equations for $\mathcal{H}$ and $F$.
Alternative techniques for such calculations are described by Couveignes~\cite{Couveignes}, Hallouin~\cite{HALLOUIN2005259}, König~\cite{Koe2017,Koe}, Malle~\cite{Malle} and Müller~\cite{Mueller2012}.

This paper is structured as follows:
Section~\ref{sec2} depicts the computation of $\mathcal{H}$ and $F$.
The computed results are verified in section~\ref{sec3} and we will show that $\mathcal{H}$ does not have $\mathbb{Q}$-rational points that lead to Galois group preserving specializations.
As a consequence we deduce that $\mathrm{PSL}_6(2)$ does not occur as the monodromy group of a rational function with rational coefficients  ramified over at least 4 points.
Furthermore we obtain explicit polynomials of degree $126$ over $\mathbb{Q}(t)$ with Galois groups isomorphic to $\mathrm{Aut}(\mathrm{PSL}_6(2))$ which are presented in section~\ref{sec4}.

\section{Computation}\label{sec2}

Let $C= (C_1,C_2,C_3,C_3)$ be the genus-0 class vector from the introduction and
$\mathcal{F}$ the family of all $\mathrm{PSL}_6(2)$-covers $f: \mathbb{P}^1 \to \mathbb{P}^1$ of degree 63 such that:
\begin{enumerate}[(i)]
\item $f$ is a 4-branch-point cover ramified over  $0,\infty,1\pm \sqrt{\lambda}$ for some $\lambda \in \mathbb{P}^1\setminus \{0,1,\infty\}$ with ramification structure:
\begin{center}
\renewcommand{\arraystretch}{1.5}
\begin{tabular}{c||c|c|c|c}
branch point & 0 & $\infty$ & $1 +  \sqrt{\lambda}$  &  $1 - \sqrt{\lambda}$ \\ \hline 
inertia class & $C_1$ & $C_2$ & $C_3$ & $C_3$
\end{tabular}
\renewcommand{\arraystretch}{1}
\end{center}
\item $f$ is normalized in the following sense: The sum of all simple roots of $f$ is 0 and the sum of all double poles is 1.
Furthermore,
$\infty$ is the unique simple pole of $f$ fixed under the action of the normalizer of the inertia group at $\infty$. Note that for any $g\in C_2$ exactly one length-1-cycle of $g$ is fixed under $N_{\mathrm{PSL}_6(2)}(\left<g\right>)$.
\end{enumerate}

\subsection{Properties of \texorpdfstring{$\mathcal{F}$}{F}}

The straight inner Nielsen class $\mathrm{SNi}^\mathrm{in}(C)$ of $C$ is the set of quadruples $(\sigma_1,\sigma_2,\sigma_3,\sigma_4) \in C_1\times C_2 \times C_3 \times C_3$ up to simultaneous conjugation satisfying both $\sigma_1 \sigma_2 \sigma_3 \sigma_4 = 1$ and $\left<\sigma_1,\sigma_2,\sigma_3,\sigma_4 \right> = \mathrm{PSL}_6(2)$.
A computer computation with Magma~\cite{Magma} yields $|\mathrm{SNi}^\mathrm{in}(C)| = 48$.

Since $\mathcal{F}$ carries the structure of an algebraic variety, various properties of $\mathcal{F}$ can be studied via the branch-point reference map:
\begin{equation} \label{Psi}
\Psi: 
\begin{cases}
\mathcal{F} \to \mathbb{P}^1_\lambda \\
\text{cover with ramification locus $\left\lbrace 0,\infty,1\pm\sqrt{\lambda}\right\rbrace$}\mapsto \lambda
\end{cases}
\end{equation}
By Riemann's existence theorem for each $\lambda_0  \in \mathbb{P}^1\setminus \{0,1,\infty\}$ and  $\sigma \in \mathrm{SNi}^\mathrm{in}(C)$ there is a unique cover (up to inner Möbius transformation) with ramification locus $(0,\infty, 1 \pm \sqrt{\lambda_0})$ and ramification $\sigma$. The normalization conditions stated in (ii) guarantee that $\mathcal{F}$ contains exactly one such cover.
As a consequence $\mathcal{F}$ is a curve and 
$\Psi$ turns out to be a Belyi map of degree $|\mathrm{SNi}^\mathrm{in}(C)| = 48$ with ramification locus $(0,1,\infty)$. The ramification of $\Psi$, denoted by  $(x,y,z)\in \mathrm{Sym}(\mathrm{SNi}^\mathrm{in}(C))^3$, is also well studied and can be calculated explicitly using the formula in \cite[Theorem III.7.8]{MM} which arises from the action of the braid group on $\mathrm{SNi}^\mathrm{in}(C)$. This triple generates a transitive group and consists of cycle structures $(6^5.4^4.2^1,7^4.4^3.3^2.2^1,2^{24})$. From this we can deduce that $\mathcal{F}$ is connected of genus $3$ (by the Riemann-Hurwitz formula).
Furthermore note that $\mathcal{F}$ can be defined over $\mathbb{Q}$ since all classes of $C$ are rational. 
 
 In the following the function field of $\mathcal{F}$ will be denoted by $\mathbb{Q}(\mathcal{F})$. The family $\mathcal{F}$ can be parameterized by a rational function
\begin{equation} \label{F}
F  = \frac{p}{q} \in \mathbb{Q}(\mathcal{F})(X)
\end{equation}
with $p,q\in \mathbb{Q}(\mathcal{F})[X]$ such that any element of $\mathcal{F}$ is obtained via specializing $F$ at some point in $\mathcal{F}$.

\subsection{Defining equations for elements in \texorpdfstring{$\mathcal{F}$}{F}} \label{defeq}
Fix $f_{\lambda_0}\in \mathcal{F}$ with $\Psi(f_{\lambda_0}) = \lambda_0$ for some $\lambda_0 \in \mathbb{P}^1\setminus \{0,1,\infty\}$. According to (i) and (ii) there exist a scalar $c_0$ and separable, monic and mutually coprime polynomials $p_7,p_{28},q_{16},q_{30},r_3,r_{20},s_3,s_{20}$ of respective degree denoted in the index such that
\begin{equation} \label{newton}
f_{\lambda_0}  = \frac{c_0 \cdot p_7 \cdot p_{28}^2}{q_{30} \cdot q_{16}^2} 
= 1+ \sqrt{\lambda_0} + \frac{c_0 \cdot r_3 \cdot r_{20}^3}{q_{30} \cdot q_{16}^2} 
= 1- \sqrt{\lambda_0} + \frac{c_0 \cdot s_3 \cdot s_{20}^3}{q_{30} \cdot q_{16}^2}
\end{equation}
where the traces of $p_7$ and $q_{16}$ are $0$ and $1$, respectively.

By comparing coefficients \eqref{newton} can be considered as a system of polynomial equations where $c_0$ and the coefficients of $p_7,p_{28},\dots, s_{20}$ are considered to be the unknowns.  This system consists of $126$ unknowns and $126$ equations, hence it is expected to have at most finitely many solutions with $f_{\lambda_0}$ being one of them.

\subsection{Walking on \texorpdfstring{$\mathcal{F}$}{F}}\label{walking} \label{newtoniteration}
Assume we are given an explicit approximative equation for $f_{\lambda_0}$, then we are able to compute another approximative equation of a cover $f_{\lambda_0+\delta}\in \mathcal{F}$ with $\Psi(f_{\lambda_0+\delta}) = \lambda_0+\delta$ for some sufficiently small $\delta\in \mathbb{C}$. This can be achieved via Newton iteration by assembling the corresponding polynomial equations similar to \eqref{newton} and using $f_{\lambda_0}$ as the initial value.

Starting from an approximative equation of a cover $f_{\mathrm{start}}\in \mathcal{F}$ we can find an approximative equation for another cover $f_{\mathrm{end}}\in \mathcal{F}$ with prescribed $\lambda_{\mathrm{end}}:= \Psi(f_{\mathrm{end}})\in \mathbb{P}^1\setminus \{0,1,\infty\}$  and  prescribed ramification $\sigma_{\mathrm{end}}\in \mathrm{SNi}^{\mathrm{in}}(C)$:

Let $\lambda_{\mathrm{start}}:=\Psi(f_{\mathrm{start}})$ and $\gamma_1$ be a path in $\mathbb{P}^1\setminus \{0,1,\infty\}$ connecting $\lambda_{\mathrm{start}}$ to $\lambda_{\mathrm{end}}$.
Lift $\gamma_1$ via $\Psi$ to $\mathcal{F}$ from a path starting in $f_{\mathrm{start}}$ and ending in some element denoted by $f_{\mathrm{end}}^*\in \mathcal{F}$, then $\Psi(f_{\mathrm{end}}^*) = \lambda_{\mathrm{end}}$. The ramification of $f_{\mathrm{end}}^*$ will be denoted by $\sigma_{\mathrm{end}}^*$. According to the ramification of $\Psi$ we can give a closed path $\gamma_{2}$ in $\mathbb{P}^1\setminus \{0,1,\infty\}$ starting in $\lambda_{\mathrm{end}}$ with the property: The lifted path of $\gamma_{2}$ in $\mathcal{F}$ via $\Psi$ connects $f_{\mathrm{end}}^*$ to another element $f_{\mathrm{end}}$ with $\Psi(f_{\mathrm{end}}) = \lambda_{\mathrm{end}}$ and ramification $\sigma_{\mathrm{end}}$.
Using Newton iteration as explained before we can slightly deform $f_{\mathrm{start}}$ at its ramification locus along $\gamma_{2}\circ \gamma_{1}$ to obtain an approximate equation for $f_{\mathrm{end}}$ having the prescribed ramification data.

\subsection{Splitting behaviour of \texorpdfstring{$\Psi$}{Psi}.}
The monodromy group of $\Psi$, generated by $x,y,z$, turns out to be imprimitive acting on 24 blocks, each of size 2.  The induced action of $(x,y,z)$ on the set $\mathcal{B}$ of these blocks, denoted by $(x',y',z')\in \mathrm{Sym}(\mathcal{B})^3$, consists of cycle structures
$(4^2.3^5.1^1$, $7^2.4^1.3^1.2^1.1^1$, $2^{12})$.
Since $(x',y',z')$ describes a genus-0 triple  the cover $\Psi$ splits as follows: 
\begin{equation} \label{split}
\Psi: \mathcal{F}  \stackrel{\Psi_{2}\phantom{.}}{\longrightarrow}  \mathbb{P}^1_\mu  \stackrel{\Psi_{24}\phantom{..}}{\longrightarrow} \mathbb{P}^1_\lambda
\end{equation}
 with a degree-2 subcover $\Psi_{2}$ and a degree-24 subcover $\Psi_{24}$ with ramification $(x',y',z')$ over $(0,1,\infty)$. The latter cover can be computed explicitly (using for example the method explained in \cite{BKW}):
\begin{equation}\label{belyi}
\lambda = \Psi_{24}(\mu) = \frac{p_{24}}{q_{24}} = 1 - \frac{r_{24}}{q_{24}}
\end{equation}
where
\begin{align*}
p_{24} &:= \left(\mu -\frac{1}{4}\right) \left(\mu^2-\frac{11}{16} \mu + \frac{1}{8}\right)^4 \left(\mu^5 -\frac{137}{4}  \mu^4 +\frac{178}{3}  \mu^3 - 34  \mu^2 + 8\mu -\frac{2}{3}\right)^3,    \\
r_{24} &:= 243 \left(\mu -\frac{1}{2}\right)^3\left(\mu -\frac{1}{3}\right)^4\left(\mu -\frac{5}{16}\right)^2 \left(\mu^2+\frac{1}{3} \mu -\frac{1}{6}\right)^7,\\
q_{24} &:= p_{24} + r_{24}.
\end{align*}

Recall that the cycle structures of $(x,y,z)$ and $(x',y',z')$ are given by
\[(6^5.4^4.2^1,7^4.4^3.3^2.2^1,2^{24})  \qquad \text{ and } \qquad (4^2.3^5.1^1,7^2.4^1.3^1.2^1.1^1, 2^{12}).\]
It is now easy to see, that these cycle structures in combination with $p_{24}$, $q_{24}$ and $r_{24}$ uniquely determine the ramification locus $\mathcal{R}_{\Psi_2} \subseteq \mathbb{P}^1_\mu$ of the degree-2 subcover $\Psi_2$.
We find $\mathcal{R}_{\Psi_2} = R_0\cup R_1\cup R_\infty $ with
\begin{align*}
R_0 :=\;& \Psi_{24}^{-1}(0)\cap \mathcal{R}_{\Psi_2} \\=\;& \left\lbrace\frac{1}{4} \right\rbrace \cup \left\lbrace \text{roots of } \mu^5 -\frac{137}{4}  \mu^4 +\frac{178}{3}  \mu^3 - 34  \mu^2 + 8\mu -\frac{2}{3}\right\rbrace,\\
R_1:=\;&\Psi_{24}^{-1}(1)\cap \mathcal{R}_{\Psi_2} = \left\lbrace  \frac{5}{16},\infty\right\rbrace ,\\
R_\infty :=\;&\Psi_{24}^{-1}(\infty)\cap \mathcal{R}_{\Psi_2} = \emptyset.
\end{align*}

\subsection{A model for \texorpdfstring{$\mathcal{F}$}{F}} \label{model}
Since $z'$ has a unique fixed point and $\mathcal{F}$ is defined over $\mathbb{Q}$ the function field analogue of \eqref{split} can be stated as 
\begin{equation} \label{KTurm}
\mathbb{Q}(\mathcal{F}) \stackrel{2}{\geq}  \mathbb{Q}(\mu) \stackrel{24}{\geq}  \mathbb{Q}(\lambda).
\end{equation}
where $\mu$ is a root of $p_{24}- \lambda q_{24}\in \mathbb{Q}(\lambda)[X]$ and
$\mathbb{Q}(\mathcal{F})$ being the degree-2 extension of $\mathbb{Q}(\mu)$ corresponding to $\Psi_2$.
The computation of $\mathcal{R}_{\Psi_2}$ guarantees the existence of a primitive element $y\in \mathbb{Q}(\mathcal{F})$, i.e.\ $\mathbb{Q}(\mathcal{F}) = \mathbb{Q}(\mu,y)$, with defining equation
\[y^2 = c P(\mu):= c\left(\mu^5 - \frac{137}{4}  \mu^4 + \frac{178}{3}  \mu^3 - 34 \mu^2 + 8 \mu - \frac{2}{3}\right)\left(\mu - \frac{1}{4}\right)\left(\mu - \frac{5}{16}\right)\]
for some square-free $c\in \mathbb{Q}$ which will be determined in \ref{compelement}.
For this reason a hyperelliptic $\mathbb{Q}$-model for $\mathcal{F}$ can be chosen to be 
\begin{equation} \label{H}
\mathcal{H} := \{(\mu,y ): y^2=cP(\mu)\}.
\end{equation}
Using this particular model $\Psi_2$ is then given by $\Psi_2(\mu,y) = \mu$ for all $(\mu,y)\in \mathcal{H}$.

\subsection{Field of definition for elements in \texorpdfstring{$\mathcal{F}$}{F}} \label{defs}

Since $\mathcal{H}$ is a model for $\mathcal{F}$, elements of $\mathcal{F}$ are obtained via specializing $F$ at points in $\mathcal{H}$.
The coefficients of a cover $f_0\in \mathcal{F}$ are therefore contained in
\begin{equation} \label{field}
\mathbb{Q}\left(\mu_0, \sqrt{cP(\mu_0)}\right) \qquad \text{where} \qquad \mu_0:=\Psi_2(f_0).
\end{equation}

The explicit computation of $\Psi_2(f_0)$ can be done in the following way:
Write $\lambda_0 := \Psi(f_0)$. Then
the fundamental group $\pi_1(\mathbb{P}^1\setminus \{0,1,\infty\},\lambda_0)$ acts on $\mathcal{B}$ and on $\Psi_{24}^{-1}(\lambda_0)$ in equivalent ways, yielding an explicitly computable bijection $\chi:\mathcal{B}\to \Psi_{24}^{-1}(\lambda_0)$ respecting the latter equivalent actions. We obtain 
\begin{equation} \label{bij}
\Psi_2(f_0) = \chi(B)
\end{equation}
whenever the ramification of $f_0$ is contained in a block $B\in \mathcal{B}$.

In particular, if an explicit cover contained in $\mathcal{F}$ is already known with algebraic numbers as coefficients, it is possible to determine the unknown rational scalar $c$.

\subsection{Obtaining elements in \texorpdfstring{$\mathcal{F}$}{F}.}\label{compelement}

By Riemann's existence theorem there exists a $\mathrm{PSL}_6(2)$-cover $h: \mathbb{P}^1 \to \mathbb{P}^1$ ramified over $(0,\infty,-1,1)$ with ramification structure $(C_3,C_3,C_1, C_2)$. Then $h^2$ turns out to be a Belyi map with ramification locus $(0,\infty,1)$ and monodromy group contained in  $ \mathrm{PSL}_{6}(2) \wr C_2 \leq S_{126}$. Its ramification consists of cycle structures $(6^{20}.2^3$, $6^{20}.2^3$, $2^{44}.1^{38})$. Using the method described in \cite{BKW}, this Belyi map of degree $126$ can be computed explicitly.
Clearly, this yields a defining (approximative) equation for $h$.

After applying suitable Möbius transformations and slightly moving the ramification points of $h$ using Newton iteration we obtain a complex approximation of a cover $f_{\mathrm{start}} \in \mathcal{F}$ with $\Psi(f_{\mathrm{start}}) =\lambda_0 := \Psi_{24}(\frac{1}{6})$.
The approach described in \ref{walking} allows the computation of a complex approximation of a cover $f_{\mathrm{end}}\in \mathcal{F}$ with $\Psi(f_{\mathrm{end}}) = \lambda_0$ and ramification contained in $B \in \mathcal{B}$ such that $\chi(B) = \frac{1}{6}$. In combination with \eqref{bij} this implies $\Psi_2(f_{\mathrm{end}}) = \frac{1}{6} \in \mathbb{Q}.$ Due to \eqref{field} the coefficients of $f_{\mathrm{end}}$ can be recognized in the quadratic number field
$\mathbb{Q}(\sqrt{cP(\Psi_2(f_{\mathrm{end}}))}) = \mathbb{Q}(\sqrt{-c \cdot 3 \cdot 7 \cdot 457})$. With the help of Magma we find $c=3$. Note that $\mathcal{H}$ from \eqref{H} is finally computed.

\subsection{Computing the universal cover \texorpdfstring{$F$}{F}.} \label{universal}
Any coefficient of $F \in \mathbb{Q}(\mathcal{F})(X) = \mathbb{Q}(\mu,y)(X)$ from \eqref{F} can be expressed as
\[
H_1 (\mu) + y H_2(\mu)
\]
where $H_1,H_2\in \mathbb{Q}(\mu)$. 
By slightly moving the ramification points of $f_{\mathrm{end}}$ via Newton iteration as described in \ref{newtoniteration} we obtain many defining equations of covers $f \in \mathcal{F}$ such that $\Psi_2(f)$ is a rational number close to $\Psi_{2}(f_{\mathrm{end}})$. Considering \eqref{field}
the coefficients of $f$ are then contained in $\mathbb{Q}(\sqrt{3P(\Psi_2(f))})$,
allowing us to read off $H_1(\Psi_2(f))$ and $H_2(\Psi_2(f))$.
Therefore, both $H_1$ and $H_2$ can be computed by interpolation. 
The resulting universal cover $F= \frac{p}{q}$ is presented in file~\texttt{3.3A}.

\begin{remark}
The standard approach of computing a hyperelliptic model $\mathcal{H}$ for $\mathcal{F}$ consists of finding a polynomial relation between $\lambda$ and a fixed coefficient of $F$ which are usually expected to generate the entire function field $\mathbb{Q}(\mathcal{F})$ with $[\mathbb{Q}(\mathcal{F}): \mathbb{Q}(\lambda)] = 48$.
This is achieved by interpolation via computing several elements $f\in\mathcal{F}$ such that $\Psi(f) \in \mathbb{Q}$ and recognizing the previously fixed coefficient as algebraic degree-48 numbers.
A Riemann-Roch space computation then leads to the hyperelliptic model $\mathcal{H}$.

Our approach takes advantage that the monodromy group of $\Psi$ is imprimitive with an explicitly computable genus-0 subcover $\Psi_{24}$. As explained in \ref{model} and \ref{compelement} this yields a defining equation for $\mathcal{H}$ after recognizing only a degree-2 number.
\end{remark}

\section{Verification and Consequences}\label{sec3}

An essential tool for the upcoming verification process is the following criterion that guarantees the existence of subgroups of a Galois group having specific properties. Similar techniques have already been applied for example by Malle, see~\cite{Malle}.

\begin{lemma}\label{lem}
Let $K$ be an arbitrary field and $f(t,X)\in K(t)[X]$ a separable and irreducible polynomial. Furthermore, let $p,q\in K[X]$ be coprime polynomials such that $p-tq\in K(t)[X]$ is separable and $f(\frac{p(t)}{q(t)},X)\in K(t)[X]$ splits nontrivially into irreducible factors of degree $d_1,\dots,d_r$. Then the following holds:
\begin{enumerate}[(a)]
\item The Galois group $\mathrm{Gal}(f\mid K(t))$ has a subgroup of index dividing $\deg(p-tq)$ with orbit lengths $d_1,\dots,d_r$.
\item If $\mathrm{Gal}(p-tq\mid K(t))$ is primitive and both $\mathrm{Gal}(p-tq\mid K(t))$ and $\mathrm{Gal}(f\mid K(t))$ have the same order, then the splitting fields of $p-tq$ and $f$ over $K(t)$ coincide.
\end{enumerate}
\end{lemma}
\begin{proof}
(a) Let $\Omega_f$ (resp.\ $\Omega_{p-tq}$) be the splitting field of $f$ (resp.\ $p-tq$) over $K(t)$ and $s$ a root of the irreducible polynomial $p-tq\in K(t)[X]$. Then $t= \frac{p(s)}{q(s)}$ and, according to the assumption,
$
f(t,X) = f(\frac{p(s)}{q(s)},X)
$
splits over $K(s)$ into irreducible factors of degree $d_1,\dots,d_r$. This also holds if we factorize $f$ over $\Omega_f \cap K(s)$. Therefore, $\mathrm{Gal}(\Omega_f \mid \Omega_f \cap K(s))\leq \mathrm{Gal}(\Omega_f \mid K(t))$ is of index dividing $[K(s):K(t)] = \deg(p-tq)$ with orbit lengths $d_1,\dots,d_r$.

(b) Recall that $f$ splits nontrivially over $K(s)$, thus $K(s)\cap \Omega_f\neq K(t)$. Since $\mathrm{Gal}(p-tq\mid K(t))$ is primitive, the latter yields $K(s)\cap \Omega_f = K(s)$, therefore $K(s)\leq \Omega_f$. Of course, the normal closure of $K(s)$ over $K(t)$ is also contained in $\Omega_f$, thus $\Omega_{p-tq}\leq \Omega_f$. Due to $|\mathrm{Gal}(p-tq\mid K(t))| = |\mathrm{Gal}(f\mid K(t))|$ we find $\Omega_{p-tq} = \Omega_f$.
\end{proof}
\begin{lemma}\label{PSLProperty}
Let $G$ be a $2$-transitive subgroup of $S_{63}$ that contains a subgroup of index dividing $63$ with orbit lengths $31$ and $32$. Then, $G$ is isomorphic to $\mathrm{PSL}_6(2)$.
\end{lemma}
\begin{proof}
This follows immediately from the classification of finite 2-transitive groups which relies on the classification of finite simple groups.
We indeed do not require such a strong result:

As explained by Dembowski~\cite[2.4.3 and 2.4.5]{Dembowski} we have $G \leq \mathrm{Aut}(\mathcal{D})$ where $\mathcal{D}$ is a symmetric $2$-$(63,31,\lambda)$-design $\mathcal{D}$ for some $\lambda \in \mathbb{N}$. An easy combinatorial consideration yields $\lambda = 15$. 
Thus, by a result of Kantor~\cite{Kantor}, $\mathcal{D}$ must be isomorphic to the projective space 
$\mathrm{PG}(5,2)$. Since $\mathrm{Aut}(\mathrm{PG}(5,2)) = \mathrm{PSL}_6(2)$ does not contain any proper 2-transitive subgroups, we conclude $G\cong \mathrm{PSL}_6(2)$.
\end{proof}

\begin{theorem} \label{thG}
Let $\mathcal{H}$ be the curve computed in \ref{model} and \ref{compelement} with defining equation
\[
 y^2 = 3\left(\mu^5 - \frac{137}{4}  \mu^4 + \frac{178}{3}  \mu^3 - 34 \mu^2 + 8 \mu - \frac{2}{3}\right)\left(\mu - \frac{1}{4}\right)\left(\mu - \frac{5}{16}\right).
\]
Furthermore, let
\[
F:=\frac{p}{q}\in \mathbb{Q}(\mathcal{H})(X) = \mathbb{Q}(\mu,y)(X)
\]
be the rational function computed in $\ref{universal}$, see ancillary file~$\mathtt{3.3A}$, and $\Psi_{24}= \frac{p_{24}}{q_{24}}$ the map from \eqref{belyi}.
Then the following holds:
\begin{enumerate}[(a)]
\item
The polynomial $p-tq$ defines a regular $\mathrm{PSL}_6(2)$-extension of $\mathbb{Q}(\mu,y,t)$. The ramification locus with respect to $t$ is given by $\mathcal{R}:=(0,\infty,1+\sqrt{\Psi_{24}(\mu)},1-\sqrt{\Psi_{24}(\mu)})$ with ramification structure $(2^{28}.1^7,2^{16}.1^{31},3^{20}.1^3,3^{20}.1^3)$.
\item
Every cover in $\mathcal{F}$ is obtained in a unique way via specialization of $F$ at some point in $\mathcal{H}$.
\end{enumerate}

\end{theorem}

\begin{proof}
(a) 
We firstly verify that $f:=p-tq$ is ramified over $\mathcal{R}$ with ramification structure $C$ from the introduction. This can be done by studying the inseparability behaviour of $f$ at the places $t\mapsto t_0$ for $t_0 \in \mathcal{R}$. The corresponding factorizations are given in the file~\texttt{3.3B}. In particular, the behaviour above  $1 \pm \sqrt{\Psi_{24}(\mu)}$ was obtained by interpolating the factorizations of several specialized polynomials. The ramification locus of $f$ cannot be larger than $\mathcal{R}$, otherwise it would contradict the Riemann-Hurwitz formula.

Let $\Omega$ be the splitting field of $p-tq$ over $\mathbb{Q}(\mu,y,t)$. Then, the geometric monodromy group
$
G:=\mathrm{Gal}(\Omega \mid (\Omega \cap \overline{\mathbb{Q}(\mu,y}))(t))$
is normal in
$A:=\mathrm{Gal}(\Omega\mid \mathbb{Q}(\mu,y,t))$. 
We now consider the specialization of $f=p-tq$ at the point $(0,\frac{1}{8}\sqrt{-10}) \in \mathcal{H}$, denoted by 
\begin{equation} \label{f0}
 f_0 = p_0-t q_0 \in \mathbb{Q}(\sqrt{-10},t)[X].
\end{equation}
Note that $f_0$ is still ramified over $4$ points. Write $\Omega_0$ for the splitting field of $f_0$ over $\mathbb{Q}(\sqrt{-10},t)$.
Then, by \cite[Theorem III.6.4]{MM} and its proof, we find
$G \cong G_0:=\text{Gal}(\Omega_0\mid (\Omega_0 \cap \overline{\mathbb{Q}})(t))$.
Using the fact that $f_0(\frac{p_0(t)}{q_0(t)},X)$ and $f_0(\frac{\overline{p_0}(t)}{\overline{q_0}(t)},X)$ split over $ \mathbb{Q}(\sqrt{-10},t)$ into irreducible factors of degree $1,62$ and $31,32$, see file~\texttt{3.3C}, Lemma~\ref{lem}(a) implies that $A_0:=\text{Gal}(\Omega_0\mid \mathbb{Q}(\sqrt{-10},t))$ must be a $2$-transitive group that contains a subgroup of index dividing $63$ with orbit lengths $31$ and $32$.
According Lemma~\ref{PSLProperty} the group $A_0$ turns out to be $\mathrm{PSL}_6(2)$. Since $A_0$ is simple and $G_0$ is normal in $A_0$ we find $G\cong G_0 \cong \mathrm{PSL}_6(2)$. As $\mathrm{PSL}_6(2)$ is also self-normalizing in $S_{63}$, we end up with $A \cong \mathrm{PSL}_6(2)$.

(b) 
We will use the following notation: For a rational function $P$ over a field of characteristic $0$ we denote by $\mathbb{Q}_P$ the field extension of $\mathbb{Q}$ generated by the coefficients of $P$.

The normalized discriminant $\delta$ of $f= p-tq$ is a polynomial in $\mathbb{Q}_F[t]$. Since the roots of $\delta$ are given by the ramification locus of $f$ its factorization in $\mathbb{Q}_F[t]$ is either of the form
$\delta
=t^k(t- (1+\sqrt{\lambda}))^\ell(t- (1 - \sqrt{\lambda }) ) ^h
$
or
$\delta
=t^k(t^2-2t+1-\lambda)^\ell$
for some $k,\ell,h\in \mathbb{N}$ where $\lambda := \Psi_{24}(\mu)$. Both cases yield $\lambda\in \mathbb{Q}_F$, therefore $\mathbb{Q}(\lambda) \subseteq \mathbb{Q}_F \subseteq \mathbb{Q}(\mu,y)$ with $[\mathbb{Q}(\mu,y):\mathbb{Q}(\lambda)] = 48$.
Fix $(\mu_0,y_0)\in \mathcal{H}$ such that $\Psi_{24}(\mu_0) = \frac{1}{2}$. Then, for the specialization of $F$ at $(\mu_0,y_0)$, denoted by $F_{(\mu_0,y_0)}$, we compute $[\mathbb{Q}_{F_{(\mu_0,y_0)}} : \mathbb{Q}] = 48$ using Magma. We end up with $\mathbb{Q}_F = \mathbb{Q}(\mu,y)$.
From the latter we see that $\mu$ and $y$ are rational functions in the coefficients of $F$. Recall that for any $\lambda_0\in \mathbb{P}^1\setminus \{0,1,\infty\}$ we find distinct points $(\mu_1,y_1),\dots,(\mu_{48},y_{48})\in \mathcal{H}$ such that $\Psi_{24}(\mu_k) = \lambda_0$ for $k=1,\dots,48$. If we specialize $F$ at these points we obtain 48 distinct $\mathrm{PSL}_6(2)$-covers
 $F_{(\mu_1, y_1)}, \dots, F_{(\mu_{48},y_{48})}$ with ramification locus $(0,\infty, 1\pm \sqrt{\lambda_0})$ and ramification structure $C$, which are all normalized with respect to inner Möbius transformations (in the sense of (ii)). Therefore, all covers $F_{(\mu_1, y_1)}, \dots, F_{(\mu_{48},y_{48})}$ lie in $\mathcal{F}$ and correspond to all distinct $48$ quadruples in $\mathrm{SNi}^{\mathrm{in}}(C)$. As a consequence, each element in $\mathcal{F}$ can be obtained uniquely via specialization.
\end{proof}

\begin{remark}
 By looking at Theorem~\ref{thG}(b) and its proof we do not get any information about the specialization behaviour of $F$ at a point $(\mu_0,y_0)\in \mathcal{H}$ with $\Psi_{24}(\mu_0) \in \{0,1,\infty\}$.
Assume, the specialization of $F$ at $(\mu_0,y_0)$, denoted by $F_{(\mu_0,y_0)}$, is a degree-$63$ cover, then one of the following cases occurs:
\begin{center}
\renewcommand{\arraystretch}{1.5}
\begin{tabular}{c||c|c}
$\Psi_{24}(\mu_0)$ & ramification locus of $F_{(\mu_0,y_0)}$ & ramification structure of $F_{(\mu_0,y_0)}$ \\ \hline\hline
$0$ & $\{0,1,\infty \}$ & contains $C_1,C_2$ \\\hline
$1$ & $\{0,2,\infty \}$ & contains $C_2,C_3$ \\\hline
$\infty$ & $\{0,\infty \}$ & only contains $C_1$ \\
\end{tabular}
\end{center}
In none of these cases $\mathrm{PSL}_6(2)$ is the monodromy group of $F_{(\mu_0,y_0)}$: Using Magma we see that
$\mathrm{PSL}_6(2)$ does not contain generating tuples of length at most 3 satisfying the product 1 condition that correspond to the respective conjugacy classes.
\end{remark}

With a little more effort we can deduce from Theorem~\ref{thG} that  $\mathrm{PSL}_6(2)$ does not occur as the monodromy group of a rational function in $\mathbb{Q}(X)$ ramified over at least $4$ points. 

In order to achieve this result, we still have to study $\mathrm{PSL}_6(2)$-covers with ramification structure $C$ and ramification locus of type $(0,\infty, \pm \sqrt{c})$. These covers can be calculated explicitly by deforming the ramification locus of covers contained in $\mathcal{F}$ via Newton iteration by assembling the defining equations explained in subsection~\ref{defeq}.

\begin{theorem}\label{thm:deg24}
Let $K$  be the degree-$24$ number field, $c\in K$ the non-square and $p,q\in K[X]$ the monic polynomials given in the ancillary file~$\mathtt{3.4A}$. Then the Galois group $A:=\mathrm{Gal}(p-tq\mid K(t))$ is isomorphic to $\mathrm{PSL}_6(2)$ in its natural $2$-transitive action on $63$ elements. The ramification structure is given by $(2^{28}.1^7,2^{16}.1^{31},3^{20}.1^3,3^{20}.1^3)$ and the ramification locus with respect to $t$ is given by $(0,\infty,\sqrt{c},-\sqrt{c})$.
\end{theorem}

\begin{proof}
In the same fashion as in the proof of Theorem~$\ref{thG}$(a) it can be calculated easily that the ramification locus of $p-tq$ is indeed given by  $(0,\infty,\sqrt{c},-\sqrt{c})$ with ramification structure $(2^{28}.1^7,2^{16}.1^{31},3^{20}.1^3,3^{20}.1^3)$, see file~$\mathtt{3.4B}$.

Let $\Omega$ be the splitting field of $p-tq$ over $K(t)$. Recall that the geometric monodromy group $G := \text{Gal}(\Omega \mid (\Omega \cap \overline{K})(t))$ is normal in $A$.
Let $\mathfrak{p} = (67,a+7)$ and $\mathfrak{q}=(67,a+42)$ be the unique prime ideals of norm $67$ in the ring of integers $\mathcal{O}_K$ of $K$ where $a$ denotes the primitive element of $K$ used in file~$\mathtt{3.4A}$.
Write $p_\mathfrak{p}$ and $q_\mathfrak{p}$ for the reduction of $p$ and $q$ modulo $\mathfrak{p}$. 
Accordingly, we define 
$A_\mathfrak{p} := \text{Gal} (\Omega_\mathfrak{p} \mid (\mathcal{O}_K/\mathfrak{p})(t))$ 
and 
$G_\mathfrak{p} := \text{Gal} ( \Omega_\mathfrak{p} \mid (\Omega_\mathfrak{p} \cap  \overline{\mathcal{O}_K/\mathfrak{p}})(t))$ with $\Omega_\mathfrak{p}$ being the splitting field of $p_\mathfrak{p}- t q_\mathfrak{p}$ over $(\mathcal{O}_K/\mathfrak{p})(t)$. Again, $G_\mathfrak{p}$ is normal in $A_\mathfrak{p}$. Of course, we will use the same notation for the reduction modulo $\mathfrak{q}$.

Note that $p_\mathfrak{p}-\frac{p_\mathfrak{p}(t)}{q_\mathfrak{p}(t)}q_\mathfrak{p}$ and $p_\mathfrak{p}-16\cdot\frac{p_\mathfrak{q}(t)}{q_\mathfrak{q}(t)}q_\mathfrak{p}$ split into irreducible factors of $1,62$ and $31,32$ over $(\mathcal{O}_K/\mathfrak{p})(t) \cong \mathbb{F}_{67}(t)$. Therefore, by Lemma~\ref{lem} and \ref{PSLProperty}, the group $A_\mathfrak{p}$ must be isomorphic to $\mathrm{PSL}_6(2)$. As $A_\mathfrak{p}$ is simple, we see that $A_\mathfrak{p}$ and $G_\mathfrak{p}$ coincide. Since $\mathfrak{p}$ is a prime of good reduction for $p-tq$, we have $G_\mathfrak{p} \cong G$ by a theorem of Beckmann, see \cite[Proposition I.10.9]{MM}. Due to the fact that $\mathrm{PSL}_6(2)$ is self-normalizing in $S_{63}$ we end up with $A = \mathrm{PSL}_6(2)$. 
\end{proof}

\begin{corollary}\label{cor}
The group $\mathrm{PSL}_6(2)$ does not occur as the monodromy group of a rational function in $\mathbb{Q}(X)$  ramified over at least $4$ points.
\end{corollary}

\begin{proof}
Suppose, there exists a rational function $f$ defined over $\mathbb{Q}$ ramified over at least $4$ points and monodromy group $\mathrm{PSL}_6(2)$. 
As $\mathrm{PSL}_6(2)$ is simple any non-trivial decomposition $f = g \circ h$ implies $\mathrm{Mon}(h) \cong \mathrm{PSL}_6(2)$, therefore we may assume that $f$ is indecomposable with primitive monodromy group.
A Magma computation shows that $C$ is the only genus-$0$ class vector of length at least $4$ containing generating tuples for $\mathrm{PSL}_6(2)$ in a primitive permutation action, thus $f$ has degree $63$ with ramification structure $C=(C_1,C_2,C_3,C_3)$.
The branch cycle lemma, see \cite[Lemma 2.8]{Voe}, asserts that the ramification locus of $f$ is of the form $(a_1,a_2,a_3,a_4)$ where $a_1,a_2 \in \mathbb{P}^1(\mathbb{Q})$ and $a_3,a_4$ fulfil a degree-2 relation over $\mathbb{Q}$.
Hence, after applying a suitable outer Möbius transformation we may assume --- without altering the field of definition --- that $f$ either has ramification locus $(0,\infty, 1\pm\sqrt{\lambda_0})$ or $(0,\infty, \pm\sqrt{\lambda_0})$ for some $\lambda_0 \in \mathbb{P}^1(\mathbb{Q})\setminus \{0,1,\infty\}$. We will now study both cases:
\begin{enumerate}
\item case $(0,\infty, 1\pm\sqrt{\lambda_0})$: Using the notation and result from Theorem~\ref{thG}(b) there exist 48 specialized covers $F_{(\mu_1,y_1)},\dots,F_{(\mu_{48},y_{48})}\in \mathcal{F}$ with $\Psi(F_{(\mu_{k},y_{k})}) = \lambda_0$ for $k \in \{1,\dots,48\}$. Up to inner Möbius transformations $f$ has to coincide with $F_{(\mu_{k},y_{k})}$ for some $k\in \{1,\dots,48\}$, therefore $F_{(\mu_{k},y_{k})}$ also has to be defined over $\mathbb{Q}$, in particular $(\mu_{k},y_{k})$ must be a $\mathbb{Q}$-rational point on $\mathcal{H}$ with $\lambda_0 = \Psi_{24}(\mu_k) \not \in \lbrace 0,1,\infty \rbrace$.

Since $\mathcal{H}$ is given by a hyperelliptic genus-$3$ model and its Jacobian is of Mordell-Weil rank 1, Chabauty's algorithm (with the implementation in Sage~\cite{sagemath} presented in \cite{Chabauty}) gives us the complete list of $\mathbb{Q}$-rational points of $\mathcal{H}$. We find $\mu_k \in \{ \frac{1}{2}, \frac{1}{3}, \frac{1}{4} , \frac{5}{16}, \infty \}$ and for all these values we see $\Psi_{24}(\mu_k) \in \lbrace 0,1 \rbrace$, a contradiction.

\item case $(0,\infty, \pm\sqrt{\lambda_0})$: After a suitable scaling process Theorem~\ref{thm:deg24} gives us 48 different $\mathrm{PSL}_6(2)$-covers $f_1,\dots,f_{48}$ that satisfy condition (ii) with ramification locus $(0,\infty,\pm 1)$. Each cover is defined over a degree-48 number field.

Since $\frac{f}{\sqrt{\lambda_0}}$ has ramification locus $(0,\infty, \pm 1)$ the cover $\frac{f}{\sqrt{\lambda_0}}$ defined over a quadratic number field has to coincide with $f_k$ for some $k\in \{1,\dots,48\}$ up to inner Möbius transformations, a contradiction.
\end{enumerate} 
This shows that $\mathrm{PSL}_6(2)$ cannot be the monodromy group of $f$.
\end{proof}

\section{Non-regular extensions of \texorpdfstring{$\mathbb{Q}(t)$}{Q(t)} with Galois group \texorpdfstring{$\text{Aut}(\mathrm{PSL}_6(2))$}{Aut(PSL(6,2))}}\label{sec4}
Although our approach does not yield a $\mathrm{PSL}_6(2)$-polynomial over $\mathbb{Q}(t)$ we at least get an explicit non-regular realization of $\mathrm{Aut}(\mathrm{PSL}_6(2))$ over $\mathbb{Q}(t)$.

\begin{theorem}
Let $f_0  \in  \mathbb{Q}(\sqrt{-10},t)[X]$ be the polynomial from \eqref{f0}, then the Galois group of $f_0 \overline{f_0}$ over  $\mathbb{Q}(t)$ is isomorphic to $\mathrm{Aut}(\mathrm{PSL}_6(2))$ in its imprimitive action on $126$ points.
\end{theorem}

\begin{proof}
The Galois groups of $f_0 = p_0-tq_0$ and $\overline{f_0} = \overline{p_0}-t\overline{q_0}$ over $\mathbb{Q}(\sqrt{-10},t)$  are isomorphic to the primitive group $\mathrm{PSL}_6(2)$.
According to Lemma~\ref{lem}(b) both $f_0$ and $\overline{f_0}$ have the same splitting field $\Omega$ over $\mathbb{Q}(\sqrt{-10},t)$ since $f_0(\frac{\overline{p_0}(t)}{\overline{q_0}(t)},X)$ is reducible over $\mathbb{Q}(\sqrt{-10},t)$. 
Let $\Omega'$ be the splitting field of $f_0\overline{f_0}$ over $\mathbb{Q}(t)$ and $G := \mathrm{Gal}( \Omega' \mid \mathbb{Q}(t))$. Clearly, $\Omega'\leq \Omega$.
Since ${f_0\overline{f_0}}$ is irreducible over $\mathbb{Q}(t)$ but obviously reducible over $\mathbb{Q}(\sqrt{-10},t)$ we find $\sqrt{-10}\in \Omega'$, therefore $\Omega' = \Omega$ and $H := \mathrm{Gal}(\Omega \mid \mathbb{Q}( \sqrt{-10},t))$ is a subgroup of $G$ with index
$
[G:H]
= [\mathbb{Q}(\sqrt{-10},t): \mathbb{Q}(t) ] = 2.
$

Let $\varphi:G\to \mathrm{Aut}(H)$ be the conjugation action of $G$ on the normal subgroup $H$, $x$ a root of $f_0$ and $y$ a root of $\overline{f_0}$. 
The point stabilizers $G_x$ and $G_y$ are conjugate in $G$ but not in $H$, because $x$ and $y$ have the same minimal polynomial over $\mathbb{Q}(t)$ but not over $\mathbb{Q}(\sqrt{-10},t)$, therefore $\varphi(H) < \varphi(G)$ and
$\mathrm{Inn}(H) =  \varphi(H) < \varphi(G) \leq \mathrm{Aut}(H).$
Since $|\mathrm{Out}(\mathrm{PSL}_6(2))| =2$ and $H \cong \mathrm{PSL}_6(2)$ this implies $\varphi(G) = \mathrm{Aut}(H)$.
In combination with $|G|= 2\cdot |H| = |\mathrm{Aut}(H)|$ we see that $\varphi$ is an isomorphism.
\end{proof}

Coincidentally, $\mathrm{PSL}_6(2)$ also happens to contain a rigid, $\mathbb{Q}(\sqrt{-7})$-rational genus-0 generating triple sharing similar properties, in particular leading to (another) non-regular $\mathrm{Aut}(\mathrm{PSL}_6(2))$-extension of $\mathbb{Q}(t)$. For the explicit realization we again apply the method explained in \cite{BKW}.

\begin{theorem}
Let $p,q \in \mathbb{Q}(\sqrt{-7})[X]$ be the polynomials of degree $63$ from the ancillary file~$\mathtt{4.2A}$.
\begin{enumerate}[(a)]
\item The polynomial $p-tq$ has Galois group $\mathrm{PSL}_6(2) \leq S_{63}$ over $\mathbb{Q}(\sqrt{-7})(t)$ with ramification locus $(0,1,\infty)$  and ramification structure ($21^3$,$4^8.2^{12}.1^7$,$2^{28}.1^7$).
\item The product $(p-tq)(\overline{p}-t\overline{q})$ has Galois group $\mathrm{Aut}(\mathrm{PSL}_6(2)) \leq S_{126}$ over $\mathbb{Q}(t)$.
\end{enumerate}

\end{theorem}

\begin{proof}
The ramification can be checked by inspecting the inseparability behaviour of $p$, $q$ and $p-q$.
A computation with Magma yields that $p-\frac{p(t)}{q(t)}q$ and $p-\frac{\overline{p}(t)}{\overline{q}(t)}q$  split in $\mathbb{Q}(\sqrt{-7},t)[X]$ into irreducible factors of degree $1$, $62$ and $31$, $32$, see file~\texttt{4.2B}. By repeating the arguments from the previous proofs both assertions follow. 
\end{proof}

\section*{Acknowledgements}

We would like to thank Joachim König for pointing out the open case $\mathrm{PSL}_6(2)$ as well as suggesting to study $\mathrm{Aut}(\mathrm{PSL}_6(2))$  over $\mathbb{Q}(t)$.
Thanks also to Stephan Elsenhans for valuable discussions about the verification process, and Peter Müller for pointing out a gap in the proof of Corollary~\ref{cor} and suggesting the alternative proof of Lemma~\ref{PSLProperty}.

\end{document}